\documentclass[a4paper,10pt]{amsart}
\usepackage{amssymb,amsthm,amsmath}
\usepackage{ifthen}
\usepackage{graphicx}

\nonstopmode
 \numberwithin{equation}{section}
\newtheorem{thm}{Theorem}[section]
\newtheorem{cor}[thm]{Corollary}
\newtheorem{lem}[thm]{Lemma}
\newtheorem{prop}[thm]{Proposition}
\newtheorem{defn}[thm]{Definition}

\theoremstyle{definition}
\newtheorem{rmk}[thm]{Remark}
\newtheorem{example}[thm]{Example}

{\qed\bigskip}

\newcounter{alphabet}

\ifx\undefined\bysame
\newcommand{\bysame}{\leavevmode\hbox to3em{\hrulefill}\,}
\fi

\begin{document}
\baselineskip=21pt
\markboth{} {}

\bibliographystyle{amsplain}
\title[Balian-Low type theorems on $L^2(\mathbb{C})$]
{Balian-Low type theorems on $L^2(\mathbb{C})$}

\author{Anirudha poria}
\author{Jitendriya Swain}
\address{Department of Mathematics,
Indian Institute of Technology Guwahati,
Guwahati 781039, \;\; India.} 
\email{anirudhamath@gmail.com, jitumath@iitg.ac.in}
\begin{abstract} In this paper we prove amalgam Balian-Low theorems and Balian-Low type theorems on $L^2(\mathbb{C})$ for the special Hermite operator using the Weyl transform.
\end{abstract}

\date{\today}
\maketitle
\def\BC{{\mathbb C}} \def\BQ{{\mathbb Q}}
\def\BR{{\mathbb R}} \def\BI{{\mathbb I}}
\def\BZ{{\mathbb Z}} \def\BD{{\mathbb D}}
\def\BP{{\mathbb P}} \def\BB{{\mathbb B}}
\def\BS{{\mathbb S}} \def\BH{{\mathbb H}}
\def\BE{{\mathbb E}}
\def\BN{{\mathbb N}}
\def\LP{{W(L^p(\BR^d, \BH), L^q_v)}}
\def\LPN{{W_{\BH}(L^p, L^q_v)}}
\def\LPQ{{W_{\BH}(L^{p'}, L^{q'}_{1/v})}}
\def\L1{{W_{\BH}(L^{\infty}, L^1_w)}}
\def\LB{{L^p(Q_{1/ \beta}, \BH)}}
\def\SP{S^{p,q}_{\tilde{v}}(\BH)}
\def\f{{\bf f}}
\def\h{{\bf h}}
\def\hp{{\bf h'}}
\def\m{{\bf m}}
\def\g{{\bf g}}
\def\ga{{\boldsymbol{\gamma}}}
\vspace{-.5cm}

\section{Introduction}
The Balian-Low theorem (BLT) is one of the fundamental and interesting result in time-frequency analysis. It says that a function $g\in L^2(\mathbb{R})$ generating Gabor Riesz basis cannot be well-localized in both time and frequency domains. Precisely if $g \in L^2(\BR)$ and if a Gabor system $\mathcal{G}(g,a,b):=\{e^{2 \pi i m bt} g(t -na)\}_{m,n \in \BZ}$ with $ab=1$, forms an orthonormal basis for $L^2(\BR)$, then \[ \left( \int_{-\infty}^\infty |t g(t)|^2 dt \right) \left( \int_{-\infty}^\infty |\gamma \hat{g}(\gamma)|^2 d\gamma \right)=+\infty, \] where $\hat{g}$ is the Fourier transform of $g$ formally defined by $\displaystyle\hat{g}(\gamma)=\int_{-\infty}^\infty g(t) e^{-2 \pi i \gamma t} dt$. This result was originally stated by Balian \cite{bal81} and independently by Low in \cite{low85}. The proofs given by Balian and Low each contained a technical gap, which was filled by Coifman et al. \cite{dau90} and extended the BLT to the case of Riesz bases. Battle \cite{bat88} provided an elegant and entirely new proof based on the operator theory associated with the classical uncertainty principle. For general Balian-Low type results, historical comments and variations of BLT we refer to \cite{ben95,dau93}. Further BLT is extended to symplectic lattices in higher dimensions \cite{fei97, gro02}, symplectic form on $\BR^{2d}$ \cite{ben03}, on locally compact abelian groups \cite{gro98}, multi-window Gabor systems by Zibulski and Zeevi \cite{zib97}  and for superframes by Balan \cite{bal98}.  We refer to \cite{asc14, ben06, gau08, hei06, nit13, tin12} for several versions of BLT under different settings.

In this paper we prove the amalgam BLTs and BLTs for the special Hermite operator using the Weyl transform on $L^2(\mathbb{C})$.
The  motivation to prove the BLT  on $L^2(\BC)$ arises from the classical Heisenberg's uncertainty principle on $L^2(\BR)$.
Let  $P$ and $M$ be the position and the momentum operators defined by $$Pf(t)=tf(t)\quad and\quad Mf(t)=\frac{d}{dt}f(t),$$ on a suitable domain.
\begin{thm}\label{th}(Classical Heisenberg's uncertainty principle on $L^2(\BR)$)
Let $f \in L^2(\BR)$. Then \begin{eqnarray}\label{hs}\Vert Pf \Vert_2 \; \Vert Mf \Vert_2\geq\frac{1}{4\pi}\|f\|_2^2.\end{eqnarray}
\end{thm}
The inequality (\ref{hs}) can also be written as \begin{eqnarray}\label{hs1}\|xf\|_2\|(-\Delta)^{\frac{1}{2}}f\|_2\geq \frac{1}{4\pi}\|f\|_2^2,\end{eqnarray} where $-\Delta$ is the Laplacian on $\mathbb{R}.$

We consider the special Hermite operator $L$ defined by \begin{eqnarray}\label{sub}L=-\frac{1}{2}(Z\bar{Z}+\bar{Z}Z)=-\Delta_z+4\pi^2|z|^2+4\pi i\left(x
\frac{\partial}{\partial y}-y\frac{\partial}{\partial x}\right),\end{eqnarray}
where \[Z=\frac{d}{dz}+ 2\pi \bar{z},\quad\quad
\bar{Z}=\frac{d}{d\bar{z}}- 2\pi z\] and $\Delta_z$  denotes the
Laplacian on $\mathbb{C}$. Analogous to the inequalities (\ref{hs}) and (\ref{hs1}), the following inequality is obtained for the special Hermite operator (see \cite{tha90}): For $f\in L^2(\mathbb{C})$
$$\Vert |z|f \Vert_2 \; \Vert L^{\frac{1}{2}}f \Vert_2\geq\frac{1}{4}\|f\|_2^2.$$
Further, the Laplacian on $\mathbb{R}$ can be written  as \begin{eqnarray}\label{lap}-\Delta=\frac{1}{4}(A^*-A)(A-A^*)=-\frac{1}{4}(A^*B+AB^*)\end{eqnarray} and satisfies the commutator relation $$[A,B]=[A,A^*]=-2I,$$where $B=A^*-A$ and $A, A^*$ denote the creation and annihilation operators on $L^2(\mathbb{R})$ respectively. For operators $X$ and $Y$, we have used the notation $[X,Y]=XY-YX$ to denote the commutator of $X$ and $Y$. The classical BLT can also be stated in terms of the operators $P$ and $M$. The expression for the special Hermite operator $L$ is identical to the Laplacian on $\mathbb{R}$ (see (\ref{sub}) and (\ref{lap})) with the relation $[Z,\bar{Z}]=-8\pi I$.

The above two facts motivate to prove the BLTs for the exact twisted Gabor frames (defined in Section \ref{3}) in terms of the operators $Z, \bar{Z}$ and $|z|, L^{\frac{1}{2}}$. We first state the Heisenberg type uncertainty inequality for $L^2(\mathbb{C})$.
\begin{thm}\label{HUP} Let $f\in L^2(\mathbb{C})$. Then $$\int_\mathbb{C}|Zf(z)|^2\,dz+\int_\mathbb{C}|\bar{Z}f(z)|^2\,dz\geq 8\pi\|f\|_2^2.$$
\end{thm}
\begin{thm}(BLT)\label{th01}
Let $g \in L^2(\BC)$. If the twisted Gabor system $\mathcal{G}^t(g,1,1)=\{T_{(m,n)}^tg:m,n\in\mathbb{Z}\}$ forms an exact frame for $L^2(\BC)$, then $\Vert Zg \Vert_2 \; \Vert \bar{Z}g \Vert_2= + \infty.$
\end{thm}
\begin{thm}(BLT)\label{th001}
Let $g \in L^2(\BC)$. If the twisted Gabor system $\mathcal{G}^t(g,1,1)$ forms an exact frame for $L^2(\BC)$, then $\Vert |z|g \Vert_2 \; \Vert L^{\frac{1}{2}} g\Vert_2= + \infty.$
\end{thm}
Further, we obtain the following version of the BLT for radial functions on $\mathbb{C}.$
\begin{cor}\label{C3}(BLT for Laguerre expansions) If $g$ is a radial function on $\mathbb{C}$ satisfying the assumptions of Theorem \ref{th01}, then $$\displaystyle\sum_{N=0}^\infty (2N+1)|g^\sharp(N)|^2=+\infty,$$ where $\displaystyle g^\sharp(N)=\frac{1}{2}\int_0^\infty g(\sqrt{s})L_N\left(\frac{s}{2}\right)\exp\left(-\frac{s}{4}\right)\,ds$ and $L_N(s)$ denotes the usual Laguerre polynomial of type 0.
\end{cor}
We also obtain the following amalgam BLTs for exact twisted Gabor frames on $L^2(\BC)$.
\begin{thm}\label{ABLT}(Amalgam BLT)
Let $g\in L^2(\mathbb{C}).$ If the twisted Gabor system $\mathcal{G}^t(g, 1,1)$ is an exact frame for $L^2(\BC)$, then $$g\not \in W(C_0,\ell^1)\quad \mbox{and}\quad W(g)\not\in \mathcal{W},$$ where $\mathcal{W}=\{T\in \mathcal{B}_2: h(z)={tr}(\pi(z)^*T) ~\mbox{and}~ h\in W(C_0,\ell^1)\}$ and $\mathcal{B}_2$, $W(g)$, $tr(\pi(z)^*T)$ denote the space of all Hilbert-Schmidt operators on $L^2(\mathbb{R})$, Weyl transform of $g$ and trace of the operator $\pi(z)^*T$ respectively.
\end{thm}As a corollary we obtain the following.
\begin{cor}\label{C1}Under the assumptions of Theorem \ref{ABLT} the following statements hold.
\begin{enumerate}
\item[(i)] The functions $L^{\frac{1}{2}}g$ and $ W(g)H^{\frac{1}{2}}$ cannot be in $W(C_0,\ell^1)$ and $\mathcal{W}$ respectively, where $H$ denotes the Hermite operator on $L^2(\mathbb{R})$.
\item[(ii)] The functions $\bar{Z}ZL^{-\frac{1}{2}}g$ and $ Z\bar{Z}L^{-\frac{1}{2}}g$ cannot be in $W(C_0,\ell^1)$.

\end{enumerate}
\end{cor}

We obtain the following version of the amalgam BLT in terms of the operators $\partial_z, \partial_{\bar{z}}, Z, \bar{Z}$ and $L^{\frac{1}{2}}.$
\begin{thm}\label{T1}Let $g$ be a Schwartz class function on $\mathbb{C}$. If any one of the following functions \begin{eqnarray}\label{1}Z g,~\partial_zg,~\partial_{\bar{z}}\bar{g},~\bar{Z}\bar{g},~L^\frac{1}{2}g\end{eqnarray} is in $W(C_0,\ell^2)$, then $\mathcal{G}^t(g,1,1)$ cannot be a twisted Gabor frame for $L^2(\mathbb{C})$.
\end{thm}

Further, the difference between the BLTs and the amalgam BLT are illustrated by Examples \ref{e1} and \ref{e2}.

The paper is organized as follows. In Section \ref{2}, we provide the necessary background for proving BLT
and discuss some basic properties of frames. In Section \ref{3}, we define the twisted Gabor frame, twisted
Zak transform and discuss the properties of twisted Gabor frames in terms of the twisted Zak transform.
Using the properties of the twisted Zak transform we prove several versions of the amalgam BLT. In Section \ref{4}, we prove the BLT for exact twisted Gabor frames on $L^2(\BC)$ using the operators
$Z,\bar{Z}$ and $|z|, L^{\frac{1}{2}}$.
In Section \ref{5}, we make use of the uncertainty principle approach to give a second prove of Theorem \ref{th01}.
We illustrate by examples to emphasize that the BLT and the amalgam BLT are two distinct results.
 Finally, we discuss several consequences of the BLT in terms of the operators $Z, \bar{Z}$ and $|z|, L$ in Remark \ref{rmkl} and prove Corollary \ref{C3}.
\section{Notations and Background}\label{2}
The Weyl transform and the twisted convolution are closely related to the Fourier transform on the Heisenberg group.
Therefore we review the representation theory on the Heisenberg group to see various objects of interest arising from it.

One of the simple and natural examples of  non-abelian, non-compact groups is the
famous Heisenberg group $\mathbf{H}$, which plays an important role in several
branches of mathematics. The Heisenberg group $\mathbf{H}$ is a unimodular nilpotent Lie group whose
underlying manifold is $\BC \times \BR $ and the group operation
is defined by $$ (z,t) \cdot (w,s) = (z+w, t+s+\frac{1}{2}
Im(z\bar{w})).$$ The Haar measure on $\mathbf{H}$ is given by $dz dt$.

By the Stone-von Neumann theorem, the only infinite dimensional unitary
irreducible representations (up to unitary equivalence) are given
by $\pi_\lambda,~\lambda \in \BR\setminus\{0\}$, where $\pi_\lambda$ is defined
by $$ \pi_\lambda(z,t) \varphi(\xi) = e^{ i\lambda t} e^{ i\lambda( x
\xi + \frac{1}{2} x  y )} \varphi(\xi + y),$$ where $ z = x + i y$
and $ \varphi \in L^2(\BR) $.

The group Fourier transform of $f \in L^1(\mathbf{H})$ is defined as $$
\hat{f}(\lambda) = \int_{\mathbf{H}} f(z,t) \pi_\lambda(z,t) dz dt, \qquad
\lambda \in \BR\setminus\{0\}.$$ Note that for each $\lambda\in \mathbb{R}\setminus\{0\}$, $\hat{f}(\lambda)$ is a bounded linear operator on $L^2(\mathbb{R}).$ Under the operation ``group convolution"  $L^1(\mathbf{H})$ turns out to be a non-commutative Banach algebra.

Let $$f^{\lambda}(z)=\int_{\mathbb{R}}e^{i\lambda t}f(z,t)dt$$
denote the inverse Fourier transform of $f$ in the $t-$variable. Therefore $\hat{f}(\lambda) = \displaystyle\int_{\mathbb{C}} f^{\lambda}(z) \pi_\lambda(z,0) dz .$
Thus we are led to consider operators of the form \begin{equation}\label{weyl} W_\lambda(g)=\int_{\mathbb{C}} g(z) \pi_\lambda(z) dz,\end{equation} where $\pi_\lambda(z,0)=\pi_\lambda(z)$. For $\lambda=4\pi$ we call (\ref{weyl}) as the Weyl transform of $g$. Thus for $g\in L^1(\BC)$ and writing $\pi(z)$ in place of $\pi_1(z)$ we have \begin{equation}\label{weyl1}W(g)\varphi(\xi)=\int_{\mathbb{C}} g(z) \pi(z)\varphi(\xi) dz, \quad\quad \varphi\in L^2(\BR).\end{equation}
For $f,g\in L^1(\mathbb{C})$, the twisted convolution of $f,g$ is defined by $$f\times g(z)=\int_{\mathbb{C}} f(z-w)g(w)e^{2\pi i Im(z\cdot\bar{w})}dw.$$
Under twisted convolution $L^1(\mathbb{C})$ is a non-commutative Banach algebra. The Weyl transform of $f$ can be explicitly written as \[W(f)\varphi(\xi)=\int_{\BC} f(z) e^{4 \pi i (x\xi+\frac{1}{2} xy)} \varphi(\xi+y) dz, \;\; \varphi \in L^2(\BR), \;\; z=x+iy,\]
which maps $L^1(\BC)$ into the space of bounded operators on $L^2(\BR)$, denoted by $\mathcal{B}$. If $f \in L^2(\BC)$, then $W(f) \in \mathcal{B}_2$, the space of all Hilbert-Schmidt operators on $L^2(\BR)$ and satisfies the Plancherel formula \[\Vert W(f)\Vert_{\mathcal{B}_2}=\frac{1}{2}\Vert f \Vert_{L^2(\BC)}.\]
 For Schwartz class functions on $\mathbb{C}$, the inversion formula for the Weyl transform is given by $$f(z)=tr(\pi(z)^*W(f)),$$ where $\pi(z)^*$ is the adjoint of $\pi(z)$ and $tr$ is the usual trace on $\mathcal{B}$.
For a detailed study on Weyl transform we refer to the text of Thangavelu \cite{tha93, tha97}.

 Let $H_k$ denote the Hermite polynomial on $\mathbb{R}$, defined by $$
H_k(x)=(-1)^k \frac{d^k}{dx^k}(e^{-x^2})e^{x^2},~k=0,1,2,\cdots,$$ and
 $h_k$ denote the normalized Hermite functions on $\mathbb{R}$
defined by
$$h_k(x)=(2^k\sqrt{\pi}k!)^{-\frac{1}{2}}H_k(x)e^{-\frac{1}{2}x^2},~~k=0,1,2,\cdots.$$
Let $A=-\frac{d}{dx}+4\pi x$ and $~A^*=\frac{d}{dx}+4\pi x$ denote the creation and
annihilation operators in quantum mechanics respectively. The (scaled) Hermite operator
$H$ is defined as
 \begin{eqnarray}\label{her}H=\frac{1}{2}(AA^*+A^*A)=-\frac{d^2}{dx^2}+16\pi ^2x^2.\end{eqnarray}
Define $\tilde{h}_k(x)=(4\pi)^\frac{1}{4}h_k(\sqrt{4\pi}x).$ The functions $\{\tilde{h}_k\}$ are the eigenfunctions of the operator $H$
with eigenvalues $4\pi(2k+1), k=0,1,2\cdots.$ Using the Hermite functions, the special Hermite functions on
$\mathbb{C}$ are defined as
\begin{equation}\label{sh}\phi_{m,n}(z)=
\sqrt{2}\int_\mathbb{R}e^{4\pi i(x\xi+\frac{1}{2}xy)}\tilde{h}_m(\xi+y)\tilde{h}_n(\xi)\,d\xi,\end{equation}
where $z=x+iy \in \mathbb{C}$ and $m,n=0,1,2\cdots$. The functions $\{\phi_{m,n}:
m,n=0,1,2\cdots\}$ form an orthonormal basis for
$L^2(\mathbb{C})$. The special Hermite functions are the eigenfunctions of the special Hermite operator $L$ (or the twisted Laplacian) with eigenvalues $4\pi(2n+1), n=0,1,2\cdots$,  defined in (\ref{sub}).
 It is easy to see that the following relations hold (see \cite{tha93, tha97}).
\begin{prop}\label{ii}
\begin{enumerate}
\item[(i)] $Z(\phi_{m,n})=i2\sqrt{\pi}\sqrt{2n}~\phi_{m,n-1}$ and $\bar{Z}(\phi_{m,n})=i2\sqrt{\pi}\sqrt{2n+2}~\phi_{m,n+1}.$
\item[(ii)] $W(Zf)=iW(f)A \quad \mbox{and}\quad W(\bar{Z}f)=iW(f)A^*$ for every Schwartz class function $f$. (This expression is similar to the relation $(\frac{d}{dx}f\hat{)~}(\gamma)=2\pi i \gamma\hat{f}(\gamma)$.)
\item[(iii)] The adjoint $Z^*$ of $Z$ is $- \bar{Z}$.
\end{enumerate}
  \end{prop}
We refer to \cite{GL, KR} for a detailed study of the vector fields $Z$ and $\bar{Z}.$

\subsection{Frame and Riesz basis}
\begin{defn}
A sequence $\{f_k : k \in \BZ\}$ is a frame for a separable Hilbert space $\mathbb{H}$ if there exist constants $A, B>0$ such that  $A \Vert f \Vert^2 \leq \sum\limits_k |\langle f,f_k \rangle|^2 \leq B \Vert f \Vert^2,$ for all $f \in \BH$.
\end{defn}
  A frame $\{f_k\}$ is exact if it ceases to be a frame when any single element $f_n$ is deleted, that is, $\{f_k\}_{k \neq n}$ is not a frame for any $n$. 
A sequence $\{f_k : k \in \BZ\}$ is called a Riesz basis for a Hilbert space $\mathbb{H}$ if there exists a continuous, invertible, linear mapping $T$ on $\mathbb{H}$ such that $\{Tf_k\}$ forms an orthonormal basis for  $\mathbb{H}$. The concept of a Riesz basis and an exact frame for a frame sequence on a separable Hilbert space are the same.
\subsection{Gabor frames and density}
For $a,b>0$, $g \in L^2(\BR^d)$ and $n,k \in \BZ^d$ define $M_{b n} g(x):= e^{2 \pi i  b n\cdot x } g(x)$ and $T_{a k}g(x):=g(x-a k)$. The collection of functions $\mathcal{G}(g, a,b)=\{ M_{b n}T_{a k}g : \; k,n \in \BZ^d\}$ in $L^2(\BR^d)$, is called a \textit{Gabor frame} or a \textit{Weyl-Heisenberg frame} if there exist constants $A, B >0$ such that
$A \Vert f \Vert^2_2 \leq \sum_{k,n \in \BZ^d} |\langle f, M_{b n}T_{a k}g \rangle |^2 \leq B \Vert f \Vert^2_2$, for all $f \in L^2(\BR^d)$.

Let $S_{\mathcal{G}}$ be the corresponding frame operator with respect to the Gabor frame $\mathcal{G}(g, a,b)$ given by 
$S_{\mathcal{G}}f:=\sum\limits_{k,n \in \BZ^d} \langle f, M_{b n}T_{a k}g \rangle M_{b n}T_{a k}g, \;\; f \in L^2(\BR^d).$
Then there exists a dual window (canonical dual window) $\tilde{g}=S_{\mathcal{G}}^{-1}(g) \in L^2(\BR^d)$ such that $\mathcal{G}(\tilde{g}, a,b)=\{M_{b n}T_{a k} \tilde{g}: k,n \in \BZ^d \}$ also constitutes a frame for $L^2(\BR^d)$, called the dual Gabor frame. Consequently every $f \in L^2(\BR^d)$ possess the expansion $f=\sum_{k,n \in \BZ^d}\langle f, M_{b n}T_{a k}g \rangle M_{b n}T_{a k}\tilde{g}  = \sum_{k,n \in \BZ^d}\langle f, M_{b n}T_{a k}\tilde{g} \rangle M_{b n}T_{a k}g$
 with unconditional convergence in $L^2(\BR^d)$.

One of the important and interesting concept in frame theory is to obtain the necessary condition on the lattice parameters $a$ and $b$,
so that the Gabor system $\mathcal{G}(g, a,b)$ constitutes a frame. The density theorem for Gabor systems provides necessary conditions for the Gabor system $\mathcal{G}(g, a,b)$
to be complete, a frame or a Riesz basis. The Gabor system $\mathcal{G}(g, a,b)$ is complete if $ab<1$, a Riesz basis if $ab=1$ and is incomplete if $ab>1.$
We refer to \cite{ben95, hei07, jan98,lan67, lan93, nit11, ram95} and the references therein for a detailed study on density related results
for both regular and irregular Gabor systems in one or higher dimensions.

\section{Twisted Zak transform and Amalgam BLT}\label{3}
From the prospective of harmonic analysis the time frequency shift of a function on a locally compact abelian group $G$ are the elements of $\hat{G},$ where $\hat{G}$ is the Pontryagin dual of $G.$ In particular, the Gabor system is defined by lattices in $G\times\hat{G}\times\hat{G}\times G$, where the time domain is $G\times\hat{G}$ and the frequency space is $\hat{G}\times G$. The twisted Gabor systems are Gabor systems for functions on $\mathbb{R}\times\hat{\mathbb{R}}$ with the twisted modulation i.e. the rotation of the standard modulation on $\mathbb{R}^2$ by an angle $\frac{\pi}{2}$.
\begin{defn}
Let $f \in L^2(\BC)$ and $\epsilon=(\epsilon_1, \epsilon_2)\in\mathbb{C}$. We define twisted translation of $f$, denoted by $T^t_\epsilon f$, as
\begin{equation}\label{eq01}
T^t_\epsilon f(z)=f(z-\epsilon)e^{2\pi i Im(z\bar{\epsilon})}=e^{2\pi i (y\epsilon_1-x\epsilon_2)} f(x-\epsilon_1, y-\epsilon_2), \;\;\; z=x+iy \in \BC,
\end{equation}where $\bar{\epsilon}$ denotes the complex conjugate of $\epsilon$ and $Im (z\bar{\epsilon})$ denotes the imaginary part of $z\bar{\epsilon}$.
\end{defn}
Further, for $f \in L^2(\BC)$ and $a,b>0$, let $\mathcal{G}^t(g, a,b)$ be the collection of functions $\{T^t_{(am,bn)}f:m,n\in\mathbb{Z}\}$, where
$T^t_{(am,bn)}f(z)=e^{2\pi i (amy-bnx)} f(x-am, y-bn), \;\;\; z=x+iy \in \BC.
$
\begin{rmk} For $a=b=1$ the operator $T^t_{(m,n)}$ has the following properties.
\begin{enumerate}
\item[(i)] The adjoint $(T^t_{(m,n)})^*$ of $T^t_{(m,n)}$ is $T^t_{(-m,-n)}$.
\item[(ii)] $T^t_{(m_1,n_1)}T^t_{(m_2,n_2)} =T^t_{(m_1+m_2,n_1+n_2)}$.
\item[(iii)] $T^t_{(m,n)}$ is a unitary operator on $L^2(\BC)$ for all $(m,n) \in \BZ^2$.
\end{enumerate}
\end{rmk}
\begin{defn}
For $a,b>0$ and $g \in L^2(\BC)$, the collection of functions $\mathcal{G}^t(g, a,b)=\{ T^t_{(am,bn)}g : \; m,n \in \BZ\}$ in $L^2(\BC)$, is called a \textit{twisted Gabor frame} or a \textit{twisted Weyl-Heisenberg frame} if there exist constants $A, B >0$ such that
\begin{equation}
A \Vert f \Vert^2_2 \leq \sum_{m,n \in \BZ} |\langle f, T^t_{(am,bn)}g \rangle |^2 \leq B \Vert f \Vert^2_2, \;\; \forall f \in L^2(\BC).
\end{equation}
\end{defn}
Then one can define the twisted Gabor tight frames, a Riesz basis and the frame operator on $L^2(\mathbb{C})$ analogously as defined for Hilbert space frames.
For $a,b>0$ and $g\in L^2(\mathbb{C})$, the system $\mathcal{G}^t(g, a,b)$  is complete in $L^2(\mathbb{C})$ if and only if the system
$\{\rho(p,q)g:(p,q)\in\Lambda\subset \mathbb{R}^4\}$ is complete in $L^2(\mathbb{R}^2)$, where $p=(am,bn), q=(bn,-am)$ and $\rho(p,q)g(x)=e^{2\pi iqx g(x-p)}$ (see \cite{ram95}).
If $ab> 1$ then the twisted Gabor system $\mathcal{G}^t(g, a,b)$ is incomplete in $L^2(\mathbb{C})$.
Therefore without loss of generality we consider the case for $a=b=1$ throughout the paper.
\subsection{Twisted Zak transform}
The Zak transform $\mathcal{Z}f$ on $L^2(\BR)$, is defined as a function of two variables by $\mathcal{Z}f(x,t)=\displaystyle\sum_{k\in\BZ}T_kf(x)\cdot M_k\textbf{1}(t), \; x,t \in \BR$,
where $\textbf{1}$ is the constant function $1$. We define the twisted Zak transform by replacing the twisted modulation instead of usual modulation, which will be an important tool
to prove our main results.
\begin{defn}
Let $f \in L^2(\BC)$. The twisted Zak transform $Z^t f$ of $f$ is the function on $\BC^2$ defined by \[ (Z^tf)(z,w)= \sum_{k \in \BZ^2} T_kf(z)\, e^{2 \pi i Im (w\bar{k})}, \;\;\; z,w \in \BC. \]
\end{defn}


Clearly  $Z^tf$ is well-defined for continuous functions with compact support and converges in $L^2-$sense for $f\in L^2(\BC)$. In fact $Z^t$ is a unitary map of $L^2(\BC)$ onto $L^2(Q \times Q)$, where $Q:=[0,1)\times [0,1).$ The idea of the proof is similar to the Zak transform on $L^2(\mathbb{R})$ as in \cite{chr03}. The unitary nature of the twisted Zak transform allows to transfer certain properties of frames for $L^2(\mathbb{C})$ into equivalent statements in terms of the twisted Zak transform on $L^2(Q \times Q)$. More precisely,  $\{f_k\}$ is  complete or a frame or an exact frame or an orthonormal basis for $L^2(\mathbb{C})$ if and only if the same is true for $\{Z^tf_k\}$ in $L^2(Q \times Q)$. As in the case of the Zak transform on $L^2(\mathbb{R})$, we obtain the following properties of the twisted Zak transform on $L^2(\BC)$.
\begin{lem}\label{lem1}
Let $f \in L^2(\BC)$. Let $z=x+iy$, $w=r+is$ and $Q:=[0,1)\times [0,1) $. Then the following holds:
\begin{enumerate}
\item[(i)]  $Z^tf(z+1,w)=e^{2 \pi i s} Z^tf(z,w)$, $Z^tf(z+i,w)=e^{-2 \pi i r} Z^tf(z,w)$ \\and $Z^tf(z,w+1)=Z^tf(z,w+i)=Z^tf(z,w).$
\item[(ii)]  $Z^t (T^t_{(m,n)}f)(z,w)=e^{2 \pi i (my-nx)} e^{2 \pi i (nr-ms)} Z^tf(z,w).$
\item[(iii)] $\mathcal{G}^t(f,1,1)$ is complete in $L^2(\BC)$ if and only if $Z^tf \neq 0$ a.e.
\item[(iv)] $\mathcal{G}^t(f,1,1)$ is minimal and complete in $L^2(\BC)$ if and only if $1/(Z^tf) \in L^2(Q \times Q).$
\item[(v)] $\mathcal{G}^t(f,1,1)$ is a frame for $L^2(\BC)$ with frame bounds $A,B$ if and only if $0 < A^{1/2} \leq |Z^tf| \leq B^{1/2} < \infty \;\;\; \mathrm{a.e.}$ In this case, $\mathcal{G}^t(f,1,1)$ is an exact frame for $L^2(\BC)$.
\item[(vi)]  $\mathcal{G}^t(f,1,1)$ is an orthonormal basis for $L^2(\BC)$ if and only if $|Z^tf|^2=1$ a.e.
\item[(vii)] $\mathcal{G}^t(f,1,1)$ is a Riesz basis for $L^2(\BC)$ with bounds $A,B$ if and only if $0 < A^{1/2} \leq |Z^tf| \leq B^{1/2} < \infty \;\;\; \mathrm{a.e.}$
\item[(viii)] If $Z^t f$ is continuous  on $\BC^2$, then $Z^t f$ has a zero in $Q\times Q$.
\end{enumerate}
\end{lem}
\begin{proof}
The proof of the lemma follows similarly as in the Zak transform for $L^2(\BR)$ (see \cite{ben95, chr03, gro01, jan88}). We only prove part (viii). Assume that $Z^tf(z,w)\neq0$ for all $(z,w)\in \mathbb{C}^2.$ Since $Z^tf$ is continuous on a simply connected domain $\mathbb{C}^2$, there is a continuous function $\varphi(z,w)$ such that $$Z^tf(z,w)=|Z^tf(z,w)|e^{2\pi i\varphi(z,w)}\quad\mbox{for}\quad(z,w)\in [0,1]^2\times [0,1]^2.$$ By part (i), we have $Z^tf(z+i,w)=e^{-2\pi i r}Z^tf(z,w)$ and $Z^tf(z,w+1)=Z^tf(z,w+i)$. Therefore for each $z$ and $w$ there are integers $l_z$ and $k_w$ such that $\varphi(z,1)=\varphi(z,i)+2\pi l_z$ and $\varphi(i,w)=\varphi(0,w)+2\pi k_w-2\pi r$. Since $\varphi(z,1)-\varphi(z,i)$ and $\varphi(i,w)-\varphi(0,w)+2\pi r$ are continuous functions of $z$ and $w$ respectively, so $l_z=l$ (say) and $k_w=k$ (say), for all $z,w\in \BC$. Therefore,
\begin{eqnarray*}
0&=&\varphi(0,1)-\varphi(0,i)+\varphi(0,i)-\varphi(i,i)+\varphi(i,i)-\varphi(i,1)+\varphi(i,1)-\varphi(0,1)\\&=&
2\pi l-2\pi k-2\pi l+2\pi k-2\pi= -2\pi,
\end{eqnarray*} a contradiction.
\end{proof}
\subsection{The Amalgam BLT}
In this section we obtain the amalgam BLT (see Theorem 3.2 of \cite{ben95}) for the Weyl transform in terms of $W(C_0,\ell^1)$ and a subspace of $\mathcal{B}_2$ using certain properties of the twisted Zak transform.
\begin{defn}
The Wiener amalgam space $W(L^p,\ell^q)$ is the Banach space of all complex-valued measurable functions $f : \BR^d \rightarrow \BC$ for which the norm
\begin{equation} \label{eqa}
\Vert f \Vert := \left(\sum_{k \in \BZ^d} \Vert f \cdot T_{k} \chi_{[0,1)^d} \Vert _p^q \right)^{1/q}  < \infty,
\end{equation}
with the obvious modification for $q= \infty.$
\end{defn}
For $p\geq 1$, we consider the amalgam space $ W(C_0,\ell^p)=\{f\in W(L^\infty,\ell^p):f \mbox{~is~ continuous}\}.$
Clearly, $ W(C_0,\ell^1)\subseteq L^1(\mathbb{R}^d)\cap L^2(\mathbb{R}^d)\cap C_0(\mathbb{R}^d).$ Now we are in a position to prove Theorem \ref{ABLT}.

\noindent{\bf Proof of Theorem \ref{ABLT}:} Suppose that $g\in W(C_0,\ell^1)$. Then by the definition of the twisted Zak transform, $Z^tg$ is continuous. By Lemma \ref{lem1} (viii), $Z^tg$ must have a zero. Therefore $|Z^tg|^{-1}$ is unbounded and by Lemma \ref{lem1} (v), $\mathcal{G}^t(g, 1,1)$ cannot be a frame.
Again assume that $\mathcal{G}^t(g, 1,1)$ is an exact frame and $W(g)\in \mathcal{W}\subset \mathcal{B}_2$. So by the inversion formula for Weyl transform  we have $g(z)={tr}(\pi(z)^*W(g))$ and $g\in W(C_0,\ell^1)$, leads to a contradiction. \qed

\noindent{\bf Proof of Corollary \ref{C1}:}
Assume that $L^{\frac{1}{2}}g\in W(C_0,\ell^1).$ Setting $g_k (z)=g(z)\cdot \chi_{[k,k+1]^2}(z)$ we have $|g_k(z)|=|L^{-\frac{1}{2}}L^{\frac{1}{2}}g_k(z)|\leq \sum_{m,n}|\langle L^{\frac{1}{2}}g_k,\phi_{m,n}\rangle|| L^{-\frac{1}{2}}\phi_{m,n}|\leq \|L^{\frac{1}{2}}g_k\|_2\leq C\|L^{\frac{1}{2}}g_k\|_\infty.$ Therefore $g\in W(C_0,\ell^1)$, contradicting to Theorem \ref{ABLT}. Further assuming $W(g)H^{\frac{1}{2}}\in \mathcal{W}$ would imply $h(z)={tr}(\pi(z)^*W(g)H^{\frac{1}{2}}) ~\mbox{and}~ h=L^{\frac{1}{2}}g\in W(C_0,\ell^1)$. This proves (i).

For part (ii), if $\bar{Z}ZL^{-\frac{1}{2}}g, Z\bar{Z}L^{-\frac{1}{2}}g\in W(C_0,\ell^1)$, then $L^{\frac{1}{2}}g\in W(C_0,\ell^1)$, contradicting to part (i).
\qed

\noindent{\bf Proof of Theorem \ref{T1}:}  Notice that the fundamental theorem of calculus for complex variables and ML - inequality hold for Schwartz class functions on $\mathbb{C}$.
 We claim that $g\in W(C_0,\ell^2)$. To prove the claim it is sufficient to show $\sum_k|g(z_k+k)|^2<\infty$, for every sequence $\{z_k\}\in[0,1]\times[0,1].$ Since $g$ is a Schwartz class function on $\mathbb{C}$, we have \begin{eqnarray}\label{e-1} \sum_k|g(z+k)|^2<\infty,~a.e. ~\mbox{on}~ [0,1]\times[0,1].\end{eqnarray} For a fixed $z_0\in [0,1]\times[0,1]$ and any sequence $\{z_k\}\in[0,1]\times[0,1]$, using (\ref{e-1}) we get \begin{eqnarray}\nonumber\left(\sum_k|g(z_k+k)|^2\right)^{\frac{1}{2}}&\leq& \left(\sum_k|g(z_k+k)-g(z_0+k)|^2\right)^{\frac{1}{2}}+\left(\sum_k|g(z_0+k)|^2\right)^{\frac{1}{2}}\\&=&
\label{partial}\left(\sum_k\left|\int_{z_0}^{z_k}\partial g(z+k)\, dz\right|^2 \right)^{\frac{1}{2}}+\left(\sum_k|g(z_0+k)|^2\right)^{\frac{1}{2}}\nonumber\\
& = &\label{Z}\left(\sum_k\left|\int_{\gamma_k}\left(Z-\frac{\bar{z}}{2}\right) g(z+k)\,dz \right|^2\right)^{\frac{1}{2}}+\left(\sum_k|g(z_0+k)|^2\right)^{\frac{1}{2}}\\
& \leq &\nonumber \sqrt{2} \left(\sum_k(M_k+m_k)^2\right)^{\frac{1}{2}}+\left(\sum_k|g(z_0+k)|^2\right)^{\frac{1}{2}}\\
&\leq&\nonumber \sqrt{2}\left[\left(\sum_k M_k^2\right)^{\frac{1}{2}}+\left(\sum_k m_k^2\right)^{\frac{1}{2}}\right]+\left(\sum_k|g(z_0+k)|^2\right)^{\frac{1}{2}}<\infty,\end{eqnarray} where
$\gamma_k$ is the straight line  joining  the points $z_0$ and $z_k$, with $$\displaystyle M_k=\mathrm{ess}\sup_{z\in \gamma_k}|Zg(z)|, ~\mbox{and}~ 2m_k=\displaystyle \mathrm{ess} \sup_{z\in\gamma_k}|zg(z)|.$$  Observe that $\displaystyle\sum_kM_k^2$ and $\displaystyle\sum_km_k^2$ are finite, since $g$ satisfies (\ref{1}). Without loss of generality we choose the curve $\gamma_k$, because the fundamental theorem of calculus for complex variables assures that the complex line integral is independent of path. Therefore $g\in W(C_0,\ell^2)$. Using this fact and the definition of twisted Zak transform yields $Z^tg$ is continuous on $\mathbb{C}$. Thus $\mathcal{G}^t(g, 1 ,1)$ cannot be a twisted Gabor frame for $L^2(\mathbb{C})$ (see Lemma \ref{lem1} (v) and (viii)).

 If $\partial_z g$ or $\partial_{\bar{z}} \bar{g}$ or $\bar{Z}\bar{g}$ is in $W(C_0,\ell^2)$, the proof follows by a similar argument with appropriate modification in $(\ref{Z})$.
 Using the fact that the operator $ZL^{-\frac{1}{2}}$ is bounded on $L^2(\mathbb{C})$ (see Theorem 2.2.2, page 37 in \cite{tha93}),
 we obtain  $\|Zg(\cdot+k)\chi_{\gamma_k}\|_2=\|ZL^{-\frac{1}{2}}L^\frac{1}{2}g(\cdot+k)\chi_{\gamma_k}\|_2\leq C\|L^\frac{1}{2}g(\cdot+k)\chi_{\gamma_k}\|_2$.
 If $L^\frac{1}{2}g \in W(C_0,\ell^2)$, then proceeding as above we get $g \in W(C_0,\ell^2)$.
 \qed
\section{Proofs of the BLTs}\label{4}
  In this section we estimate the oscillations of the twisted Zak transform on $ L^2(\mathbb{C})$ in terms of $\|Zf\|_2$ and $\|\bar{Z}f\|_2$ using the Weyl transform to prove Theorem \ref{th01} and Theorem \ref{th001}. But
  unlike the Fourier transform of functions in $L^1(\mathbb{R})$, the Weyl transform of functions in $L^1(\mathbb{C})$ are bounded operators on $L^2(\mathbb{R})$.
  So estimating the oscillations in this case is not easy as compare to the estimates obtained in the real variables (see \cite{ben95,dau90}).
  After estimating the oscillations of the twisted Zak transform on $ L^2(\mathbb{C})$, we proceed along the similar idea used in \cite{ben95,dau90} in our setup to prove our main results. We start with the following lemma.
\begin{lem}\label{ineq}
Let $f,Zf$ and $\bar{Z}f\in L^2(\mathbb{C}).$ Fix $\epsilon=(\epsilon_1,\epsilon_2)\in \BC$. If $\tilde{f}(z)=f(z)e^{2\pi i(y\epsilon_1-x\epsilon_2)}$ and $\tau_\epsilon f(z)=f(z-\epsilon)$, then there exists $C>0$ such that
\begin{enumerate}
\item[(i)]  $\|\tau_\epsilon f-f\|_2\leq C|\epsilon|\left(\|Zf\|_2+\|\bar{Z}f\|_2+(1+|\epsilon|)\|f\|_2\right).$
\item[(ii)] $\|T^t_\epsilon f-f\|_2\leq C|\epsilon|\left(\|Zf\|_2+\|\bar{Z}f\|_2+(1+|\epsilon|)\|f\|_2\right).$
\item[(iii)] $\|\tilde{f}-f\|_2\leq C|\epsilon|\left(\|Zf\|_2+\|\bar{Z}f\|_2+(1+|\epsilon|)\|f\|_2\right)$.
\end{enumerate}
\end{lem}
\begin{proof}
In order to prove (i) we first calculate $W(\tilde{f}).$
 For $\phi\in L^2(\mathbb{R})$, we consider \begin{eqnarray}\label{til}\nonumber W(\tilde{f})\phi(\xi)&=&\displaystyle\int_{\mathbb{C}}f(z)e^{2\pi i(y\epsilon_1-x\epsilon_2)}e^{4\pi i(x\xi+\frac{1}{2}xy)}\phi(\xi+y)\,dxdy \\&=& \nonumber e^{-2\pi i\xi\epsilon_1}\int_{\mathbb{C}}f(z)e^{4\pi i(x(\xi-\epsilon_2/2)+\frac{1}{2}xy)}[e^{(2\pi i(\xi+y)\epsilon_1)}\phi(\xi+y)]\,dxdy\\&=& \pi(-\epsilon_1,0)W(f) \pi(\epsilon_1,0)\phi(\xi).\end{eqnarray}
Since
 $$W(T^t_\epsilon f)\phi(\xi)=\displaystyle\int_{\mathbb{C}}f(z-\epsilon)e^{2\pi i(y\epsilon_1-x\epsilon_2)}e^{4\pi i(x\xi+\frac{1}{2}xy)}\phi(\xi+y)\,dxdy,$$

 substituting $w=z-\epsilon$ and using (\ref{til}), we get \begin{eqnarray}\label{trans} W(T^t_\epsilon f)\phi(\xi)&=&  \pi(\epsilon_1,\epsilon_2)W(\tilde{f}) \phi(\xi).\end{eqnarray}
Again $$ W(\tau_\epsilon f)\phi(\xi)=\displaystyle\int_{\mathbb{C}}f(z-\epsilon)e^{4\pi i(x\xi+\frac{1}{2}xy)}\phi(\xi+y)\,dxdy,$$
 substituting $w=z-\epsilon/2,$ we have \begin{eqnarray}\nonumber W(\tau_\epsilon f)\phi(\xi)&=&\pi\left(\frac{\epsilon_1}{2},\frac{\epsilon_2}{2}\right)W(T^t_\frac{\epsilon}{2}f)\phi(\xi). \end{eqnarray}
Using (\ref{trans}) and  (\ref{til}), we obtain \begin{eqnarray}\nonumber W(\tau_\epsilon f)\phi(\xi)&=& \pi\left(\frac{\epsilon_1}{2},\frac{\epsilon_2}{2}\right)\pi\left(\frac{\epsilon_1}{2},\frac{\epsilon_2}{2}\right)\pi(-\epsilon_1,0)W(f)\pi(\epsilon_1,0) \phi(\xi)\\&=& e^{-2\pi i\epsilon_1\epsilon_2}W(f)\pi(\epsilon_1,0)\phi(\xi+\epsilon_2).\end{eqnarray}

Therefore \begin{eqnarray}\label{esti 1}\nonumber|W(\tau_\epsilon f)\phi(\xi)-W(f)\phi(\xi)|&=&|e^{-2\pi i\epsilon_1\epsilon_2}W(f)\pi(\epsilon_1,0)\phi(\xi+\epsilon_2)-W(f)\phi(\xi)|\\&\leq&
C|\epsilon||W(f)\pi(\epsilon_1,0)\phi(\xi+\epsilon_2)|+|W(f)(\pi(\epsilon_1,0)\phi(\xi+\epsilon_2)-\phi(\xi))|.\end{eqnarray}
By taking $g$ such that $\hat{g}=\phi$ and applying mean value theorem, we have \begin{eqnarray*}\pi(\epsilon_1,0)\phi(\xi+\epsilon_2)-\phi(\xi)&=&\widehat{(\tau_{2\epsilon_1}g)}(\xi+\epsilon_2)-\hat{g}(\xi)\\&=&2\epsilon_1 \widehat{\frac{dg}{d\xi}}(\xi+\epsilon_2+2\theta\epsilon_1)+\epsilon_2\frac{d\hat{g}}{d\xi}(\xi+\theta'\epsilon_2),
\end{eqnarray*} for some $\theta,\theta'\in (0,1).$
 Then there exists constants $C_1,C_2,C_3>0$ such that\begin{eqnarray}\label{esti 2}\nonumber|W(f)(\pi(\epsilon_1,0)\phi(\xi+\epsilon_2)-\phi(\xi))|&\leq& \nonumber C_1|\epsilon|^2|W(f)\phi(\xi+\epsilon_2+2\theta\epsilon_1)|\\&+&\nonumber C_2|\epsilon||W(f)(A^*+A)\phi(\xi+\epsilon_2+2\theta\epsilon_1)|\\&+&C_3|\epsilon||W(f)(A^*-A)\phi(\xi+\theta'\epsilon_2)|.\end{eqnarray} Computing $L^2-$norm of $W(\tau_\epsilon f-f)\phi$ using (\ref{esti 1}) and (\ref{esti 2}), we get
\begin{eqnarray}\label{esti 3}&&\nonumber\|W(\tau_\epsilon f)\phi-W(f)\phi\|_2\\&\leq& C|\epsilon|(\|W(f)\pi(\epsilon_1,0)\phi\|_2+|\epsilon|\|W(f)\phi\|_2+\|W(f)A^*\phi\|_2+\|W(f)A\phi\|_2).\end{eqnarray}
Since the operator $\phi\mapsto \pi(\epsilon_1,0)\phi$ is unitary on $L^2(\mathbb{R})$, calculating the Hilbert-Schmidt norm of $W(\tau_\epsilon f-f)$ using the orthonormal bases $\{\phi_{m,n}\}$ (defined in (\ref{sh})) and $\{\pi(\epsilon_1,0)\phi_{m,n}\}$ (only for the first term after the inequality in (\ref{esti 3})), we obtain
\begin{eqnarray*}\|\tau_\epsilon f-f\|_2 = 2\|W(\tau_\epsilon f-f)\|_{\mathcal{B}_2}\leq C|\epsilon|((1+|\epsilon|)\|f\|_2+\|Zf\|_2+\|\bar{Z}f\|_2).
\end{eqnarray*}
In order to prove (ii), we use (\ref{trans}) and (\ref{til}) to get \begin{eqnarray*}W(T^t_\epsilon f)\phi(\xi)&=&\pi(\epsilon_1,\epsilon_2)\pi(-\epsilon_1,0)W(f) \pi(\epsilon_1,0) \phi(\xi)\\&=&e^{-2\pi i\epsilon_1\epsilon_2}W(f) \pi(\epsilon_1,0) \phi(\xi+\epsilon_2).\end{eqnarray*} Proceeding exactly as in part (i), we get
\begin{eqnarray}\label{zz}\|T^t_\epsilon f-f\|_2 = \|W(T^t_\epsilon f-f)\|_{\mathcal{B}_2}\leq C|\epsilon|((1+|\epsilon|)\|f\|_2+\|Zf\|_2+\|\bar{Z}f\|_2).
\end{eqnarray}
Finally for the part (iii) we use (\ref{zz}) and part (i) to get \begin{eqnarray*}\|\tilde{f}-f\|_2\leq \|\tilde{f}-T^t_\epsilon f\|_2+\|T^t_\epsilon f -f\|_2\leq C|\epsilon|((1+|\epsilon|)\|f\|_2+\|Zf\|_2+\|\bar{Z}f\|_2).\end{eqnarray*}
\end{proof}
We use the following notation to estimate the upper bound for the oscillation of the twisted Zak transform over small cubes. Let $x=(t,w)\in \mathbb{R}^2$ and $r>0$. Then $Q(x;r)$ is the square centered at $x$ with radius $r$, i.e. \begin{eqnarray*}Q(x;r)& = & \left[t-\frac{r}{2},t+\frac{r}{2}\right]\times\left[w-\frac{r}{2},w+\frac{r}{2}\right]\\&=& \left\{(u,v): u\in\left[t-\frac{r}{2},t+\frac{r}{2}\right], v\in \left[w-\frac{r}{2},w+\frac{r}{2}\right] \right\}.\end{eqnarray*} Thus the square $Q=[0,1)\times[0,1)$ can be represented as $Q(\frac{1}{2},\frac{1}{2};1)$.
\begin{thm}\label{est} Let $f,Zf$, $\bar{Z}f\in L^2(\mathbb{C}),~G=Z^tf,~\alpha_0=(z_0,w_0)\in Q(z_0;1)\times Q(w_0;1):=Q[\alpha_0,1], z_0\in[-\frac{3}{2},\frac{3}{2}]\times[-\frac{3}{2},\frac{3}{2}]$ and $w_0,\epsilon \in \mathbb{C}$ be given. Then
$$\|T_{\epsilon,1}^tG-G\|_{L^2(Q[\alpha_0,1])}\leq C|\epsilon|\left(\|Zf\|_2+\|\bar{Z}f\|_2+(1+|\epsilon|)\|f\|_2\right),$$ and $$\|T_{\epsilon,2}^tG-G\|_{L^2(Q[\alpha_0,1])}\leq C|\epsilon|(\|Zf\|_2+\|\bar{Z}f\|_2+(1+|\epsilon|)\|f\|_2),$$ where $T_{\epsilon,j}^tG(z,w)$ is the twisted translation of $G$ in the $j$th variable for $j=1,2.$
\end{thm}
\begin{proof}
Notice that $T_{\epsilon,1}^tG(z,w)=e^{2\pi i Im(z\bar{\epsilon})}Z^t(\tau_{\epsilon}f)(z,w)$. Then by using the fact that the twisted Zak transform $Z^t$ is an unitary operator of $L^2(\mathbb{C})$ onto $L^2(Q[\alpha_0,1])$, we get
\begin{eqnarray*}
&& \|T_{\epsilon,1}^tG-G\|_{L^2(Q[\alpha_0,1])}\\&\leq& \|T_{\epsilon,1}^tG-Z^t(\tau_{\epsilon}f)\|_{L^2(Q[\alpha_0,1])}+\|Z^t(\tau_{\epsilon}f)-G\|_{L^2(Q[\alpha_0,1])}\\
 &=&\|(e^{2\pi i Im(z\bar{\epsilon})}-1)Z^t(\tau_{\epsilon}f)\|_{L^2(Q[\alpha_0,1])}+\|Z^t(\tau_{\epsilon}f)-Z^tf\|_{L^2(Q[\alpha_0,1])}\\
 &\leq&8\pi |\epsilon|\|f\|_2+\|\tau_{\epsilon}f-f\|_2\\
 &\leq& C|\epsilon|\left(\|Zf\|_2+\|\bar{Z}f\|_2+(1+|\epsilon|)\|f\|_2 \right),
 \end{eqnarray*} by Lemma \ref{ineq} (i).
 But \begin{eqnarray*}
 T_{\epsilon,2}^tG(z,w)&=& G(z,w-\epsilon)e^{2\pi i Im(w\bar{\epsilon})}=\sum_{k\in \mathbb{Z}^2} f(z-k)e^{2\pi i Im((w-\epsilon)\bar{k})}e^{2\pi i Im(w\bar{\epsilon})}\\
 &=& e^{-2\pi i Im(\bar{z}\epsilon)}\sum_{k\in\mathbb{Z}^2}f(z-k)e^{2\pi i Im((\overline{z-k})\epsilon)}e^{2\pi i Im(w\bar{k})}e^{2\pi i Im(w\bar{\epsilon)}}\\
 &=&  e^{2\pi i Im((w+z)\overline{\epsilon})}\sum_{k\in\mathbb{Z}^2}h(z-k)e^{2\pi i Im(w\bar{k})}\\
 &=& e^{2\pi i Im((w+z)\overline{\epsilon})}Z^th(z,w),
 \end{eqnarray*} where $h(z)=f(z)e^{2\pi i Im(\bar{z}\epsilon)}=f(z)e^{-2\pi i Im(z\bar{\epsilon})}.$
 Therefore \begin{eqnarray}\nonumber\|T_{\epsilon,2}^tG-G\|_{L^2(Q[\alpha_0,1])}&\leq& \|Z^t(h-f)\|_{L^2(Q[\alpha_0,1])}+\|(e^{2\pi i Im((w+z)\overline{\epsilon})}-1)Z^tf\|_{L^2(Q[\alpha_0,1])}\\
  &\leq& \label{nu}\|h-f\|_2+2\pi|\epsilon|\|Z^tf\|_{L^2(Q[\alpha_0,1])}.\end{eqnarray}
 Applying Lemma \ref{ineq} (iii) for $\epsilon=(-\epsilon_1,-\epsilon_2)$ in (\ref{nu}), we get $$\|T_{\epsilon,2}^tG-G\|_{L^2(Q[\alpha_0,1])}\leq C|\epsilon|(\|Zf\|_2+\|\bar{Z}f\|_2+(1+|\epsilon|)\|f\|_2).$$ 
\end{proof}
\begin{cor}\label{cor}
Let $f,Zf,\bar{Z}f\in L^2(\mathbb{C}),~G=Z^tf$. For $0<r<1,$ fix $~\alpha_0=(z_0,w_0)\in Q(z_0,r)\times Q(w_0,r):=Q[\alpha_0,r], z_0\in[-\frac{3}{2},\frac{3}{2}]\times[-\frac{3}{2},\frac{3}{2}]$ and $w_0,\epsilon \in \mathbb{C}$. Then
\begin{eqnarray}\label{a}\int_{Q[\alpha_0,r]}|T_{\epsilon,1}^tG(z,w)-G(z,w)|dzdw\leq r^2|\epsilon|C_{1,f}^{\epsilon}(r)\end{eqnarray} and \begin{eqnarray}\label{b}\int_{Q[\alpha_0,r]}|T_{\epsilon,2}^tG(z,w)-G(z,w)|dzdw\leq r^2|\epsilon|C_{2,f}^{\epsilon}(r),\end{eqnarray}
 where $T_{\epsilon,j}^tG(z,w)$ is the twisted translation of $G$ in the $j$th variable with $\displaystyle\lim_{r\to 0} C_{j,f}^\epsilon(r)=0$ for $j=1,2$.
\end{cor}
\begin{proof}
For $z=x+iy\in\mathbb{C}$ and $r>0$, we denote $E(z,r)$ as the set $E(z,r)=\bigcup_{m,n\in \mathbb{Z}}[x-\frac{r}{2}+n, x+\frac{r}{2}+n]\times[y-\frac{r}{2}+m, y+\frac{r}{2}+m].$
Then for $f\in L^2(\mathbb{C})$, we have $Z^tf(z,w)\cdot\chi_{E(z,r)}=Z^t(f\cdot\chi_{E(z,r)})(z,w)$. Proceeding as in the proof of Theorem \ref{est}, we get
$$\|T_{\epsilon,1}^tG-G\|_{L^2(Q[\alpha_0,r])}
 \leq  C|\epsilon|\left(\|Zf\cdot\chi_{E(z_0,r)}\|_2+\|\bar{Z}f\cdot\chi_{E(z_0,r)}\|_2+(1+|\epsilon|)\|f\cdot\chi_{E(z_0,r)}\|_2\right)$$
and $$\|T_{\epsilon,2}^tG-G\|_{L^2(Q[\alpha_0,r])} \leq C|\epsilon|(\|Zf\cdot\chi_{E(z_0,r)}\|_2+\|\bar{Z}f\cdot\chi_{E(z_0,r)}\|_2+(1+|\epsilon|)\|f\cdot\chi_{E(z_0,r)}\|_2).$$ 
Applying Cauchy-Schwartz inequality in the left hand side of (\ref{a}) and (\ref{b}), the proof follows with
$$C_{1,f}^{\epsilon}(r)=C_{2,f}^{\epsilon}(r)=C|\epsilon|\left(\|Zf\cdot\chi_{E(z_0,r)}\|_2+\|\bar{Z}f\cdot\chi_{E(z_0,r)}\|_2+(1+|\epsilon|)\|f\cdot\chi_{E(z_0,r)}\|_2\right). $$
Further, using the fact that $\|f.\chi_{E(z_0,r)}\|_2\to 0$ as $r\to 0$, we have $\displaystyle\lim_{r\to 0}C_{j,f}^\epsilon(r)=0$ for $j=1,2$.
\end{proof}
Now we are in a position to prove Theorem \ref{th01}.

\noindent{\bf Proof of Theorem \ref{th01}:}
Assume that $\{T^t_{(m,n)}g\}$ is an exact frame for $L^2(\BC)$. Then by Lemma \ref{lem1} (v), we have
$0 < A^{1/2} \leq |Z^tg| \leq B^{1/2} < \infty \;\;\mathrm{a.e.}$ Assume that both $Zg$ and $\bar{Z}g \in L^2(\BC).$ We will show our assumption together with Lemma \ref{lem1} (v)
 leads to a contradiction in the following three steps.

\noindent{\bf Step 1:} (Construction of an continuous averaged function $G_r(z,w)$ that approximating $G(z,w)=Z^tg(z,w)$.)
Let $\rho$ be a non-negative infinitely differentiable function with compact support on $\mathbb{R}^4$ with $\int_{\mathbb{R}^4}\rho(\zeta) d\zeta=1.$ For $r>0,$ let $\rho_r(\zeta)=r^{-4}\rho(\frac{\zeta}{r})$. Then $\displaystyle\rho_r(\zeta)$ is an approximate identity for $L^1(\mathbb{C}^2)$ with respect to the twisted convolution (see Lemma 5, page 242 of \cite{Mail}). Without loss of generality we assume that the support of $\rho$ is contained in $[0, 1)^4.$
For $z,w\in \mathbb{C}$, we define $$G_r(z,w)=G\times \rho_r(z,w)=\int_{[0,1)^4}G(z-z',w-w')\rho_r(z',w')e^{2\pi i Im(z\bar{z'}+w\bar{w'})}\,dz'dw'.$$ Then $G_r$ satisfies the following properties:
\begin{enumerate}
\item[(a)] $|G_r(z_1,w_1)-G_r(z_2,w_2)|\leq 2\left(\pi (r+\max\{|z_1|, |w_1|\}) +\frac{1}{r}\right)B^\frac{1}{2}\left(|z_1-z_2|+|w_1-w_2|\right)$.
\end{enumerate}
By Lemma \ref{lem1} (v), we have \begin{eqnarray*}
&&|G_r(z_1,w_1)-G_r(z_2,w_2)|\\&=& \dfrac{1}{r^4}\left|\int_{Q[z_1^*,w_1^*;r]}G(u,v)e^{2\pi i Im(z_1\bar{u}+w_1\bar{v})}\,dudv-\int_{Q[z_2^*,w_2^*;r]}G(u,v)e^{2\pi i Im(z_2\bar{u}+w_2\bar{v})}\,dudv\right|\\
&\leq& \dfrac{1}{r^4}\left|\int_{Q[z_1^*,w_1^*;r]}G(u,v)\left[e^{2\pi i Im(z_1\bar{u}+w_1\bar{v})}-e^{2\pi i Im(z_2\bar{u}+w_2\bar{v})}\right]\,dudv\right|\\
&+& \dfrac{1}{r^4}\left|\int_{Q[z_1^*,w_1^*;r]}G(u,v)e^{2\pi i Im(z_2\bar{u}+w_2\bar{v})}\,dudv-\int_{Q[z_2^*,w_2^*;r]}G(u,v)e^{2\pi i Im(z_2\bar{u}+w_2\bar{v})}\,dudv\right|\\
&\leq& \frac{2 \pi}{r^4}B^\frac{1}{2}\left(|z_1-z_2|(r+|z_1|)+|w_1-w_2| (r+|w_1|)\right)|Q[z_1^*,w_1^*;r]| \\ &+& \frac{1}{r^4}B^\frac{1}{2}\int_{Q[z_1^*,w_1^*;r]\Delta Q[z_2^*,w_2^*;r]}dudv\\&\leq&
2 \pi B^\frac{1}{2}(r+\max\{|z_1|, |w_1|\}) \left(|z_1-z_2|+|w_1-w_2|\right)+ \frac{1}{r^4}B^\frac{1}{2}|Q[z_1^*,w_1^*;r]\Delta Q[z_2^*,w_2^*;r]|\\&\leq& 2\left(\pi (r+\max\{|z_1|, |w_1|\}) +\frac{1}{r}\right)B^\frac{1}{2}\left(|z_1-z_2|+|w_1-w_2|\right),
\end{eqnarray*}
where $\Delta$ is the symmetric difference operator and $Q[z_j^*,w_j^*;r]=Q[z_j-\frac{r}{2}(1+i),w_j-\frac{r}{2}(1+i);r]$, $j=1,2.$
\begin{enumerate}\item[(b) (i)] $G_r(z,w+1)=G_r(z,w)+\psi_{1,r}(z,w)$ and $G_r(z,w+i)=G_r(z,w)+\psi_{2,r}(z,w)$,
\item[(ii)] $G_r(z+1,w)=e^{2\pi i Im(w)}G_r(z,w)+\psi_{3,r}(z,w)$ and $G_r(z+i,w)=e^{-2\pi i Im(iw)}G_r(z,w)+\psi_{4,r}(z,w)$, where $|\psi_{j,r}(z,w)|\leq 2\pi B^{1/2}r,~j=1,2,3,4.$
\end{enumerate}
\begin{eqnarray*}G_r(z,w+1)&=&\int_{[0,1)^4}G(z-z',w+1-w')\rho_r(z',w')e^{2\pi i Im(z\bar{z'}+(w+1)\bar{w'})}\,dz'dw'\\
&=& \int_{[0,1)^4}G(z-z',w-w')\rho_r(z',w')e^{2\pi i Im(z\bar{z'}+w\bar{w'})}\,dz'dw'+\psi_{1,r}(z,w)\\&=&
G_r(z,w)+\psi_{1,r}(z,w).
\end{eqnarray*}where $\psi_{1,r}(z,w)=\displaystyle\int_{[0,1)^4}(e^{2\pi i Im(\bar{w'})}-1)G(z-z',w-w')\rho_r(z',w')e^{2\pi i Im(z\bar{z'}+w\bar{w'})}\,dz'dw'.$

Further \begin{eqnarray*}|\psi_{1,r}(z,w)|&=& \left|\int_{[0,1)^4}(e^{2\pi i Im(\bar{w'})}-1)G(z-z',w-w')\rho_r(z',w') e^{2\pi i Im(z\bar{z'}+w\bar{w'})}\,dz'dw'\right|\\&\leq&
B^{1/2}\int_{[0,1)^4}|2\pi Im(\bar{w'})|\rho_r(z',w')\,dz'dw'\\&\leq& 2\pi B^{1/2}r.
\end{eqnarray*}
Similarly we can obtain the other identities with $|\psi_{j,r}(z,w)|\leq 2\pi B^{1/2}r,~j=2,3,4.$
\begin{enumerate}
\item[(c)] Fix $(z,w),(z',w')\in \mathbb{C}^2$ and using (a) one has \begin{eqnarray*} && |G(z,w)-G_r(z,w)| \geq |G(z,w)|-|G_r(z,w)-G_r(z',w')|-|G_r(z',w')|\\&\geq & A^{\frac{1}{2}}-2B^\frac{1}{2}\left(\pi (r+\max\{|z|, |w|\}) +\frac{1}{r}\right)(|z-z'|+|w-w'|)-|G_r(z',w')|.
    \end{eqnarray*}\end{enumerate}
In particular, for a fixed $(z,w)\in[0,1)^4, c<1$ and $(z,w)\in Q[z',w';cr]$, we have \begin{eqnarray}\label{lower}&& |G(z,w)-G_r(z,w)| \nonumber \\ &\geq &
A^{\frac{1}{2}}-2B^\frac{1}{2}cr\left(\pi (r+\max\{|z|, |w|\}) +\frac{1}{r}\right)-|G_r(z',w')|.\end{eqnarray}

\noindent{\bf Step 2:} For any $(z_0,w_0)\in[0,1)^4,c<1$ and $r<1$, we have
\begin{eqnarray}\label{both} c^4r^4(A^{\frac{1}{2}}-2crc_{z,w}^rB^\frac{1}{2}-|G_r(z',w')|)&\leq&\int_{Q[z,w;cr]}|G(z,w)-G_r(z,w)|dzdw\\
&\leq& c^2r^4 C(r),\nonumber\end{eqnarray} where $c_{z,w}^r=\left(\pi (r+\max\{|z|, |w|\}) +\frac{1}{r}\right)$ and $C(r)$ is independent on the point $(z,w)$ and \begin{eqnarray}\label{cr}\lim\limits_{r\to 0}C(r)=0.\end{eqnarray}
\begin{eqnarray*}
&&\int_{Q[z_0,w_0;cr]}|G(z,w)-G_r(z,w)|\,dzdw\\&\leq& \int_{[0,1)^4}|\rho_r(z',w')|\int_{Q[z_0,w_0;cr]}|G(z,w)-G(z,w-w')e^{2\pi i Im(w\bar{w'})}|dz'dw'dzdw\\&+& \int_{[0,1)^4}|\rho_r(z',w')|\int_{Q[z_0,w_0;cr]}|G(z,w-w') - G(z-z',w-w')e^{2\pi i   Im(z\bar{z'})}|dz'dw' dz dw \\&\leq& \int_{[0,1)^4}|\rho_r(z',w')|\int_{Q[z_0,w_0;cr]}|G(z,w)-G(z,w-w')e^{2\pi i Im(w\bar{w'})}|dz dw dz'dw'\\&+& \int_{[0,1)^4}|\rho_r(z',w')|\int_{Q[z_0,w_0-w';cr]}|G(z,s)-G(z-z',s)e^{2\pi i   Im(z\bar{z'})}|dzds dz'dw'.
\end{eqnarray*}
Using Corollary \ref{cor} and the fact that $|z'|<1, \; |w'|<1$, we have $|C_{1,g}^{z'}(r)|< C_{1,g}(r)$ and $|C_{2,g}^{w'}(r)|< C_{2,g}(r)$, where
\begin{eqnarray*}C_{1,g}(r)=C_{2,g}(r)&=&C|\epsilon|\left(\|Zg\cdot\chi_{E(z_0,r)}\|_2+\|\bar{Z}g\cdot\chi_{E(z_0,r)}\|_2+(1+|\epsilon|)\|g\cdot\chi_{E(z_0,r)}\|_2\right).\end{eqnarray*}
Putting $C(r)=C_{1,g}(r)+C_{2,g}(r)$, we get (\ref{cr}). Then the inequality (\ref{both}) can be obtained  by (\ref{lower}) and applying Cauchy-Schwartz inequality in the last term of the above calculation.

\noindent {\bf Step 3:} Claim: $\displaystyle\inf_{(z,w)\in [0,1)^4}|G(z,w)|=0$.

  From (\ref{both}), we get $|G_r(z',w')|\geq A^{\frac{1}{2}}-2crc_{z,w}^rB^\frac{1}{2}-\frac{C(r)}{c^2}$. Choose $c<1$ such that
  $A^{\frac{1}{2}}-2crc_{z,w}^rB^\frac{1}{2}>\frac{A^{\frac{1}{2}}}{2}$ and letting $r\to 0$ we get $|G_r(z,w)|\geq \frac{A^{\frac{1}{2}}}{2}.$ Since $G_r(z,w)$ is continuous real valued function on $[0,1)^4$ (see \cite{rr55}, pp. 377-385), there exists a continuous real valued function $\theta_{r}$ such that $G_r(z,w)=|G_r(z,w)|e^{i\theta_{r}(z,w)}.$ Define
\begin{center}$\delta_{1,r}(z,w)=1+\dfrac{\psi_{1,r}(z,w)}{G_r(z,w)},$\\$\delta_{2,r}(z,w)=1+\dfrac{\psi_{2,r}(z,w)}{G_r(z,w)},$\\$\delta_{3,r}(z,w)=1+\dfrac{\psi_{3,r}(z,w)}{e^{2\pi i Im(w)}G_r(z,w)},$\\$\delta_{4,r}(z,w)=1+\dfrac{\psi_{4,r}(z,w)}{e^{-2\pi i Im(iw)}G_r(z,w)}.$\end{center}
 Clearly $\delta_{j,r}$ is continuous and non-vanishing on $[0,1]^4$ for each $r>0$ and every $j=1,2,3,4.$ So there exists a continuous real valued function $\theta_{j,r}$ such that $\delta_{j,r}(z,w)=|\delta_{j,r}(z,w)|e^{i\theta_{j,r}(z,w)}$ $~~\mbox{for}~~j=1,2,3,4.$
 Since \begin{center}$G_r(z,w+1)=G_r(z,w)\delta_{1,r}(z,w)$\\$G_r(z,w+i)=G_r(z,w)\delta_{2,r}(z,w)$\\$G_r(z+1,w)=e^{2\pi i Im(w)}G_r(z,w)\delta_{3,r}(z,w)$\\$G_r(z+i,w)=e^{-2\pi i Im(iw)}G_r(z,w)\delta_{4,r}(z,w),$\end{center}
for each $r>0$ and for all $z,w\in [0,1]\times[0,1]$, there are integers $I_r, J_r, K_r$ and $L_r$ such that \begin{center}$\theta_r(z,1)=\theta_r(z,0)+\theta_{1,r}(z,0)+2\pi I_r,$\\$\theta_r(z,i)=\theta_r(z,0)+\theta_{2,r}(z,0)+2\pi J_r$,\\$\theta_r(1,w)=2\pi  Im(w)+\theta_r(0,w)+\theta_{3,r}(0,w)+2\pi K_r$,\\$\theta_r(i,w)=-2\pi  Im(iw)+\theta_r(0,w)+\theta_{4,r}(0,w)+2\pi L_r.$\end{center}
Now \begin{eqnarray*}
0&=&[\theta_r(0,1)-\theta_r(0,i)]+[\theta_r(0,i)-\theta_r(i,i)]+[\theta_r(i,i)-\theta_r(i,1)]+[\theta_r(i,1)-\theta_r(0,1)]\\&=&
[\theta_{1,r}(0,0)-\theta_{2,r}(0,0)+2\pi (I_r-J_r)]+[-\theta_{4,r}(0,i)-2\pi L_r]\\&+& [\theta_{2,r}(i,0)-\theta_{1,r}(i,0)+2\pi (J_r-I_r)]+[-2\pi  +\theta_{4,r}(0,1)+2\pi L_r]\\
&=&[\theta_{1,r}(0,0)-\theta_{2,r}(0,0)-\theta_{4,r}(0,i)+\theta_{2,r}(i,0)-\theta_{1,r}(i,0)-2\pi  +\theta_{4,r}(0,1).
\end{eqnarray*} Letting $r\to 0$ we get $0=-2\pi$, a contradiction.
\qed

Observe that if $|z|f,~ L^{\frac{1}{2}}f\in L^2(\mathbb{C})$, then using the bounds $\|h-f\|_2\leq2\pi |\epsilon|\||z|f\|_2$ in (\ref{nu}) and  $\|Zf\|_2+\|\bar{Z}f\|_2\leq 4\|L^{\frac{1}{2}}f\|_2$ in Lemma \ref{ineq}, the bounds in Theorem \ref{est} can be expressed in terms of $f, |z|f$ and $L^{\frac{1}{2}}f$.

{\noindent \bf Proof of Theorem \ref{th001}:}
By the above observation, the proof of theorem follows similarly as in the proof of Theorem \ref{th01}.\qed
\section{Uncertainty Principle and the BLT}\label{5}
We start this section with the proof of Theorem \ref{HUP}. Then we prove the weaker version of BLT for the exact twisted Gabor frame and establish the equivalence of the BLT and  the weak BLT.\\
\noindent{\bf Proof of Theorem \ref{HUP}:} Let $f\in L^2(\mathbb{C})$. Recall that $\{\phi_{m,n}:m,n\in\mathbb{N}\cup\{0\}\}$ (defined in (\ref{sh})) forms an orthonormal basis for $L^2(\mathbb{C})$.  Then we have $f(z)=\displaystyle\sum_{m,n=0}^\infty\langle f,\phi_{m,n}\rangle\phi_{m,n}(z)$. By Proposition \ref{ii}, we get
$$\|Zf\|_2^2 +\|\bar{Z}f\|_2^2= 4\pi\sum_{m,n=0}^\infty(4n+2)|\langle f,\phi_{m,n}\rangle|^2\geq 8\pi\sum_{m,n=0}^\infty|\langle f,\phi_{m,n}\rangle|^2=8\pi\|f\|_2^2.
$$
Using the fact that $Z\phi_{m,0}=0$ for $m=0,1,2,\cdots,$ we conclude that equality holds in the above inequality if and only if $n=0$ i.e. $f=\displaystyle\sum_{m=0}^\infty c_m\phi_{m,0}$.
\qed

\begin{thm}(Weak BLT)\label{th1}
Assume $g \in L^2(\BC)$ is such that $\mathcal{G}^t(g,1,1)$ is an exact twisted Gabor frame for $L^2(\BC)$ and $\tilde{g}$ be the dual window function. Then the functions $Zg, Z\tilde{g}, \bar{Z}g$ and $\bar{Z}\tilde{g} $ cannot be in $ L^2(\BC)$ simultaneously, i.e.  $\Vert Zg \Vert_2 \; \Vert Z\tilde{g} \Vert_2 \; \Vert \bar{Z}g \Vert_2 \; \Vert \bar{Z}\tilde{g} \Vert_2= + \infty.$
\end{thm}
\begin{proof}
 Assume that $Zg,Z\tilde{g}, \bar{Z}g, \bar{Z}\tilde{g}\in L^2(\BC).$ Since $\mathcal{G}^t(g,1,1)$ is a twisted Gabor frame for $L^2(\BC)$, any $f \in L^2(\BC)$ can be expressed as $f=\sum\limits_{m,n} \langle f,T^t_{(m,n)}g\rangle \tilde{g}_{m,n}=\sum\limits_{m,n}\langle f, \tilde{g}_{m,n} \rangle T^t_{(m,n)}g$, where $\tilde{g}_{m,n}=T^t_{(m,n)} \tilde{g}$. Since the commutators of $Z$ and $\bar{Z}$ with $T^t_{(m,n)}$ satisfy $[T^t_{(m,n)}, Z]= -(m-in) \pi T^t_{(m,n)}$ and $[T^t_{(m,n)}, \bar{Z}]= (m+in)\pi T^t_{(m,n)}$, and $\{T^t_{(m,n)}g\}$ is bi-orthogonal to $\{\tilde{g}_{m,n}\}$, we get
\begin{eqnarray}\label{no}
\langle Zg, Z\tilde{g} \rangle &=& \sum_{m,n} \langle Zg, \tilde{g}_{m,n} \rangle \langle T^t_{(m,n)}g, Z\tilde{g} \rangle \nonumber\\
&=& \sum_{m,n} \langle T^t_{(-m,-n)}g, \bar{Z}\tilde{g} \rangle  \langle \bar{Z} g, \tilde{g}_{-m,-n} \rangle\nonumber\\
&=& \sum_{m,n} \langle \bar{Z} g, \tilde{g}_{m,n} \rangle \langle T^t_{(m,n)}g, \bar{Z}\tilde{g} \rangle \nonumber\\
&=& \langle \bar{Z}g, \bar{Z} \tilde{g} \rangle .
\end{eqnarray}
Therefore, by the bi-orthogonality relation (Proposition 5.4.8 of \cite{chr03}, pp. 101) and the above equality gives \[ 1= \langle g, \tilde{g} \rangle = - \frac{1}{8\pi} \langle g, [Z,\bar{Z}]\tilde{g} \rangle=-\frac{1}{8\pi}(\langle Zg, Z\tilde{g}\rangle -\langle \bar{Z}g, \bar{Z} \tilde{g} \rangle)  =0, \] a contradiction.
\end{proof}
\begin{rmk}
If the twisted Gabor frame $\mathcal{G}^t(g,1,1)$ forms an orthonormal basis then $g=\tilde{g}$ and the above theorem is precisely the analogue of Battle's proof of BLT in \cite{bat88}.
\end{rmk}
 In order to prove the equivalence of the BLT and the weak BLT it is enough to show $\bar{Z} g \in L^2(\BC) \Leftrightarrow Z\tilde{g} \in L^2(\BC) \;\; \mathrm{and} \;\; Zg  \in L^2(\BC) \Leftrightarrow \bar{Z} \tilde{g} \in L^2(\BC)$, whenever $\mathcal{G}^t(g,1,1)$ is an exact frame for $L^2(\mathbb{C}).$ To achieve this goal we need the following proposition.

 \begin{prop}\label{pro1}
If $g \in L^2(\BC)$ and $\mathcal{G}^t(g,1,1)$ is an exact twisted Gabor frame for $L^2(\mathbb{C})$, then the relation $Z^t \tilde{g}=1/{\overline{Z^t g}}$ holds for the dual window $\tilde{g} \in L^2(\BC)$.
\end{prop}
\begin{proof}
 Let $h={Z^t}^{-1}\left(\frac{1}{\overline{Z^t g}}\right)$. By Lemma \ref{lem1} (v), $0 < A^{1/2} \leq |Z^tg| \leq B^{1/2} < \infty$ a. e. on $Q\times Q$. Therefore $h$ is well defined and $h\in L^2(\mathbb{C})$. Let $z=x+iy$ and $w=r+is \in \BC.$ Using Lemma \ref{lem1} (ii) and the bi-orthogonality relation (Proposition 5.4.8 of \cite{chr03}, pp. 101), we have
\begin{eqnarray*}
\langle h,T^t_{(m,n)}g \rangle = \langle Z^t h,Z^t T^t_{(m,n)}g \rangle &=& \int_{Q\times Q} \frac{1}{\overline{Z^t g}(z, w)}e^{-2 \pi i (my-nx)} e^{ -2 \pi i (nr-ms)} \overline{Z^t g}(z,w)\;dzdw \\
&=& \langle \tilde{g}, T^t_{(m,n)}g \rangle, \quad \forall\;\; m,n\in \mathbb{Z}.
\end{eqnarray*} Since $\mathcal{G}^t(g,1,1)$ is complete in $L^2(\BC)$ and $h, \tilde{g} \in L^2(\BC)$, it follows that $h=\tilde{g}$.
\end{proof}
\begin{thm}\label{th2}
If $g \in L^2(\BC)$ and $\mathcal{G}^t(g,1,1)$ is an exact twisted Gabor frame for $L^2(\mathbb{C})$, then
\[
\bar{Z} g \in L^2(\BC) \Leftrightarrow Z\tilde{g} \in L^2(\BC) \;\; \mathrm{and} \;\; Zg  \in L^2(\BC) \Leftrightarrow \bar{Z} \tilde{g} \in L^2(\BC).
\]
\end{thm}
\begin{proof}
Assume that $Zg\in L^2(\mathbb{C}).$ Then
\begin{eqnarray}\label{eq12}
Z^t (Zg)(z,w) &=& \sum_{k} Zg(z-k) e^{2 \pi i Im(w.\bar{k})} \nonumber \\
&=& \sum_{k} \left(\frac{d}{dz}+ 2\pi \bar{z}\right)g(z-k) e^{2 \pi i Im(w.\bar{k})} \nonumber \\
&=& \partial_z (Z^tg)(z,w) + 2\pi \bar{z}(Z^tg)(z,w) -  2\partial_w (Z^tg)(z,w).
\end{eqnarray}
Similarly,
$Z^t (\bar{Z}g)(z,w)=\partial_{\bar{z}} (Z^tg)(z,w) - 2\pi z(Z^tg)(z,w) - 2\partial_{\bar{w}} (Z^tg)(z,w)$.
Now using Proposition \ref{pro1}, we compute
\begin{eqnarray}\label{eq14}
\overline{Z^t (Z \tilde{g})(z,w)} &=& \overline{\partial_z (Z^t\tilde{g})(z,w)} + 2\pi z \overline{(Z^t\tilde{g})(z,w)} -  2\overline{ \partial_w (Z^t \tilde{g})(z,w)} \nonumber \\
&=& \partial_{\bar{z}} \left( 1/Z^t g \right)(z,w)+\frac{2\pi z}{(Z^tg)(z,w)}- 2\partial_{\bar{w}}(1/Z^tg)(z,w)\nonumber \\
&=& - \frac{\partial_{\bar{z}}(Z^tg)(z,w)}{(Z^tg)^2(z,w)} + 2\pi z \frac{(Z^tg)(z,w)}{(Z^tg)^2(z,w)}+  \frac{2\partial_{\bar{w}}(Z^tg)(z,w)}{(Z^tg)^2(z,w)} \nonumber \\
&=& - \frac{\partial_{\bar{z}}(Z^tg)(z,w)- 2\pi z (Z^tg)(z,w)- 2\partial_{\bar{w}}(Z^tg)(z,w)} {(Z^tg)^2(z,w)}\nonumber \\
&=&  - \frac{Z^t(\bar{Z}g)(z,w)}{(Z^tg)^2(z,w)}.
\end{eqnarray}
Thus it follows that $\bar{Z}g  \in L^2(\BC) \Leftrightarrow Z \tilde{g} \in L^2(\BC)$, provided all the calculations are justified in the sense of distributions. The other equivalent relation can be obtained by a similar argument.
\end{proof}
\begin{rmk}\label{rmkl}
Let $g \in L^2(\BC)$ be such that $\mathcal{G}^t(g,1,1)$ is an exact twisted Gabor frame for
$L^2(\mathbb{C})$. Then the following statements hold.
\begin{itemize}
\item[$(i)$] The function $L^{\frac{1}{2}}g$ cannot be in $L^2(\mathbb{C})$: If
$L^{\frac{1}{2}}g\in L^2(\mathbb{C})$, then
$\|L^{\frac{1}{2}}g\|_2^2=\langle L^{\frac{1}{2}}g,
L^{\frac{1}{2}}g\rangle=\langle g,
Lg\rangle=\frac{1}{2}(\|Zg\|_2^2+\|\bar{Z}g\|_2^2).$
Therefore
$L^{\frac{1}{2}}g\in L^2(\mathbb{C}) \Leftrightarrow Zg,\bar{Z}g\in
L^2(\mathbb{C})$, leads to a contradiction to Theorem \ref{th01}.
\item[$(ii)$] The functions $Z\bar{Z}g$ and $\bar{Z}Zg$ cannot both be in
$L^2(\mathbb{C})$: If $Z\bar{Z}g,\bar{Z}Zg\in L^2(\mathbb{C})$, then
$Lg\in L^2(\BC) \quad \mathrm{and~this~ would~ imply} \quad
L^{\frac{1}{2}}g=L^{-\frac{1}{2}}(Lg)\in L^2(\BC) $ contradicting to Theorem
\ref{th01}.
\item[$(iii)$] The operators $R$ and $\bar{R}$ (Riesz transforms)
defined by $Rg=Z L^{-\frac{1}{2}}g \quad \mathrm{and} \quad \bar{R}g=\bar{Z}
L^{-\frac{1}{2}}g.$ Then the functions $\bar{Z}Rg$ and $Z\bar{R}g$ cannot both be in
$L^2(\mathbb{C})$: If $\bar{Z}Rg, Z\bar{R}g \in L^2(\BC)$, then
$L^{\frac{1}{2}}g=-\frac{1}{2}(\bar{Z}Z+Z\bar{Z})L^{-\frac{1}{2}}g=-\frac{1}{2}(\bar{Z}R+Z\bar{R})g
\in L^2(\BC),$
leading to a contradiction.
\item [$(iv)$] By AM-GM inequality, $\|Zg\|_2\|\bar{Z}g\|_2=+\infty$ would imply $\|L^{\frac{1}{2}}g\|_2=+\infty$. So Theorem \ref{th01} implies Theorem \ref{th001}. But there exists a function $g\in L^2(\mathbb{C})$ such that $\||z|g\|_2\|L^{\frac{1}{2}}g\|_2=+\infty$ but $\|Zg\|_2\|\bar{Z}g\|_2<+\infty.$ : Let $h(z)=\sum_{j,k=2}^\infty \frac{1}{jk}\phi_{j,k}(z)$ and $g(z)=L^{-\frac{1}{2}}h(z)$. Then clearly $h\in L^2(\mathbb{C}).$ Since Riesz transforms are bounded operators on $L^2(\mathbb{C})$, $\|Zg\|_2\|\bar{Z}g\|_2<+\infty.$ Now we show that $\||z|g\|_2=+\infty.$ Then by
    Theorem 1.3.3. page 17-18 of \cite{tha93}, we get $zg(z)=zL^{-\frac{1}{2}}h(z)=\sum_{j,k=2}^\infty \frac{1}{jk}zL^{-\frac{1}{2}}\phi_{j,k}(z)=\sum_{j,k=2}^\infty \frac{1}{jk\sqrt{2k+1}}z\phi_{j,k}(z)=\sum_{j,k=2}^\infty \frac{1}{jk\sqrt{2k+1}}i[\sqrt{2j}\phi_{j-1,k}(z)-\sqrt{2k+1}\phi_{j,k+1}(z)].$ So $\||z|g\|_2^2=2\sum_{j,k=2}^\infty \frac{1}{jk^2(2k+1)}+\sum_{j,k=2}^\infty \frac{1}{j^2k^2}=+\infty.$
    \item [$(v)$] By Proposition \ref{pro1} and Lemma \ref{lem1} (v), the system $\mathcal{G}^t(\tilde{g}, 1, 1)$ is also an exact twisted Gabor frame. Therefore the statements $(i), (ii)$ and $(iii)$ also hold if $g$ is replaced by the dual window $\tilde{g}$.
\item [$(vi)$] BLT and the Hermite operator: The BLTs can also be established for the Hermite operator (see (\ref{her})) i.e. for $g\in L^2(\mathbb{R})$, if the Gabor system $\mathcal{G}(g, 1, 1)$ forms an exact frame for $L^2(\mathbb{R})$, then $\|xg\|_2\|H^{\frac{1}{2}}g\|_2=\|Ag\|_2\|A^*g\|_2=+\infty.$
\end{itemize}
\end{rmk}
\noindent{\bf Proof of Corollary \ref{C3}}:
If $g$ is a radial function (i.e. $g(z)=g(|z|)$) on $\mathbb{C}$, then the Weyl transform reduces to the Laguerre transform and $W(g)=2\pi\sum_{N=0}^\infty R_N(g)P_N,$
 where $R_N(g)=\displaystyle\int_0^\infty g(s)L_N(\frac{s^2}{2})\exp(-\frac{s^2}{4})s\,ds$,  $L_N(s)$ is the usual Laguerre polynomial of type 0
 and $P_N$ is the projection of $L^2(\mathbb{R})$ onto the $N$th eigenspace spanned by the normalized Hermite function $h_N$ (see \cite{tha97, tha90}).
 Since $W(L^{\frac{1}{2}}g)=W(g)H^\frac{1}{2}$, we have $W(L^{\frac{1}{2}}g)=2\pi\sum_{N=0}^\infty (2N+1)^{\frac{1}{2}}R_N(g)P_N.$
 So $\|W(L^{\frac{1}{2}}g)\|_{\mathcal{B}_2}^2=4\pi^2\displaystyle\sum_{N=0}^\infty (2N+1)|R_N(g)|^2$.
  Define $g^\sharp(N)=\frac{1}{2}\int_0^\infty g(\sqrt{s})L_N(\frac{s}{2})\exp(-\frac{s}{4})\,ds=R_N(g)$. By the Plancherel formula for the Weyl transform we have $\|L^{\frac{1}{2}}g\|_2^2=16\pi^2\sum_{N=0}^\infty (2N+1)|g^\sharp(N)|^2$.
 By part $(i)$ of Remark \ref{rmkl} we get $\sum_{N=0}^\infty (2N+1)|g^\sharp(N)|^2=+\infty$.
\qed

The BLT and the amalgam BLT are two distinct results. The following examples illustrate the difference between the BLT and the amalgam BLT.
\begin{example}\label{e1}
Let $f: \BR^2 \to \BR$ defined by
$ f(x,y)=
\begin{cases}
      e^{-\left[ \frac{1}{x(1-x)} +\frac{1}{y(1-y)} \right]}, & (x,y) \in
(0,1) \times (0,1), \\
      0, & \mathrm{otherwise}.
\end{cases}
$
Let $z = x+iy$ and define $g: \BC \to \BR$ by
$ g(z)=\displaystyle\sum_{(k_1, k_2) \in \BN^2 } \frac{1}{k_1^{\frac{3}{2}}
k_2^{\frac{3}{2}}} f(x-k_1, y-k_2). $
Then clearly $g \in W(C_0, \ell^1).$

Further, $W(g)=\displaystyle\sum_{(k_1, k_2) \in \BN^2 } \frac{1}{k_1^{\frac{3}{2}}
k_2^{\frac{3}{2}}} Wf(x-k_1, y-k_2). $
Clearly $W(g) \in \mathcal{B}_2$. From the inversion formula for the Weyl
transform it follows that $W(g) \in \mathcal{W}.$
Next we show that $\|\bar{Z}g \|_2=+\infty.$ Consider
\begin{eqnarray*}
&& \|\bar{Z}g \|^2_2 = \int_{\BC} |\bar{Z}g(z)|^2 dz \\
& \geq & \sum_{m, n \in \BN} \int_{[m,m+1] \times [n,n+1]} \bar{Z}g(z) \cdot
\overline{\bar{Z}g(z)} \; dz \\
&=& \sum_{m,n \in \BN} \int_{[m,m+1] \times [n,n+1]} \left[ \left|
\frac{dg(z)}{d\bar{z}} \right|^2 +4\pi^2 |zg(z)|^2- 2\pi\mathrm{Re} \left(\bar{z}
g(z) \frac{dg(z)}{d\bar{z}} \right) \right] dz.
\end{eqnarray*}
Note that for each $m, n\in\mathbb{N}$ and $(x,y)\in (m,m+1)\times (n,n+1)$, the integrand
\begin{align*}
 & \left|
\frac{dg(z)}{d\bar{z}} \right|^2 - 2\pi\mathrm{Re} \left(\bar{z} g(z)
\frac{dg(z)}{d\bar{z}} \right)   \\
&= \frac{1}{4m^3 n^3}
e^{-\left[ \frac{2}{x(1-x)} +\frac{2}{y(1-y)} \right]} \Bigg{[}
\frac{(2x-1)^2}{x^4(1-x)^4} + \frac{(2y-1)^2}{y^4(1-y)^4} +\frac{4\pi(2x-1)}{x(1-x)^2} +\frac{4\pi(2y-1)}{y(1-y)^2}
\Bigg{]} \geq 0.
\end{align*}
Therefore
\begin{align*}
\|\bar{Z}g \|^2_2  & \geq 4\pi^2 \sum_{m, n \in \BN} \int_{[m,m+1] \times
[n,n+1]} |zg(z)|^2 \; dz \\
& \geq 4\pi^2 \sum_{m, n \in \BN} \int_{[m,m+1] \times [n,n+1]}
\frac{m^2+n^2}{m^3 n^3} |f(x-m,y-n)|^2 \; dxdy \\
& \geq 4\pi^2\|f\|^2_2 \sum_{m \in \BN} \frac{1}{m} = +\infty.
\end{align*}
Since $\|L^\frac{1}{2}g\|^2_2=\frac{1}{2}(\|Zg \|^2_2+\|\bar{Z}g \|^2_2)$, $L^\frac{1}{2}g\notin L^2(\mathbb{C}).$
So there exists a function on $L^2(\mathbb{C})$ such that $\|Zg\|_2\|\bar{Z}g\|_2=\|L^{\frac{1}{2}}g\|_2\||z|g\|_2=+\infty$, but $g\in W(C_0,\ell^1)$ and $W(g) \in \mathcal{W}.$ That means the BLT does not imply the amalgam BLT.
\end{example}
The following example shows that the amalgam BLT does not imply the BLT.
\example \label{e2} We construct a function $f$ such that $Zf$ and $\bar{Z}f\in L^2(\mathbb{C})$ but $f\not \in W(C_0,\ell^1)$ and $W(f)\not\in \mathcal{W}.$ For sufficiently large $k$ (say $k>N$) choose  $a_k\neq b_k$ such that $[a_k-\frac{1}{k},b_k+\frac{1}{k}]\subset [k,k+1]$ and $b_k^3-a_k^3<k$. Define the continuous function $g_k$ by $$
g_k(x)=
\begin{cases}
\dfrac{1}{\log k}(x-a_k+\frac{1}{k}), & x\in[a_k-\frac{1}{k},a_k], \\
\dfrac{1}{k\log k}, & x\in [a_k,b_k],\\
\dfrac{1}{\log k}(b_k+\frac{1}{k}-x), & x\in[b_k,b_k+\frac{1}{k}], \\
0, & x\not \in [a_k-\frac{1}{k},b_k+\frac{1}{k}].
\end{cases}
$$
Clearly the function $g=\displaystyle \sum_{k=N}^\infty g_k$ is continuous on $\mathbb{R}$.
Also $\|g\|_2\leq \displaystyle2\sum_{k=N}^\infty \dfrac{1}{(k\log k)^2}<\infty$, $\|xg\|_2\leq \displaystyle3\sum_{k=N}^\infty \dfrac{1}{k(\log k)^2}<\infty$,
and $\|g'\|_2\leq \displaystyle2\sum_{k=N}^\infty \dfrac{1}{k(\log k)^2}<\infty$, where $g'$ is the classical derivative of $g$, defined except at countably many points.
 Define $f(z)=f(x,y)=g(x)g(y)$. Since $Zf=\frac{1}{2}(f_x-if_y)+2\pi(xf-iyf)$, we have
\begin{eqnarray*}\|Zf\|_2&\leq &\frac{1}{2}(\|f_x\|_2+\|f_y\|_2)+2\pi(\|xf\|_2+\|yf\|_2)  \\&= &\frac{1}{2}(\|g'\|_2\|g\|_2+\|g'\|_2\|g\|_2)+2\pi(\|xg\|_2\|g\|_2+\|yg\|_2\|g\|_2)<\infty.
\end{eqnarray*} Similarly $\|\bar{Z}f\|_2<\infty.$ So $L^\frac{1}{2}f$ and $|z|f$ are in $L^2(\mathbb{C})$. We have
$\| f\|_{W(L^\infty,\ell^1)} = \sum_{k \in \BZ^2} \Vert f \cdot T_{k} \chi_{[0,1)^2} \Vert _\infty  = \sum_{k_1,k_2=N}^\infty \dfrac{1}{k_1\log k_1}\dfrac{1}{k_2\log k_2}=+\infty.$
If $W(f)\in \mathcal{W}$, then the inversion formula for Weyl transform  gives $f\in W(C_0,\ell^1)$, which is a contradiction. Therefore $W(f)\notin \mathcal{W}$.

\section*{Acknowledgments}
The first author wishes to thank the Ministry of Human Resource Development, India for the  research fellowship and
Indian Institute of Technology Guwahati, India for the support provided during the period of this work.

\end{document}